\documentclass{article}

\usepackage[english]{babel}
\usepackage{latexsym}
\usepackage{amsmath}
\usepackage{amsthm}
\usepackage{amssymb}
\usepackage{indentfirst}
\usepackage{amsfonts}
\usepackage{stackrel}
\usepackage{dsfont}
\usepackage{tikz}
\usepackage{slashbox}
\usepackage[mathscr]{eucal}
\usepackage{authblk}

\newtheorem{theorem}{Theorem}
\newtheorem{acknowledgement}[theorem]{Acknowledgement}

\newtheorem{lemma}[theorem]{Lemma}

\newcommand{\reals}{\mathbb{R}}
\newcommand{\naturals}{\mathbb{N}}

\newcommand{\eps}{\varepsilon}
\newcommand{\W}{\mathcal{W}}

\begin{document}

\title{Parametric critical point theorems and their applications to boundary value problems on the Sierpi\'{n}ski Gasket}
\author{Marek Galewski\ \&\ Mateusz Krukowski}
\affil{\L\'od\'z University of Technology, Institute of Mathematics, \\ W\'ol\-cza\'n\-ska 215, \
90-924 \ \L\'od\'z, \ Poland}
\maketitle

\begin{abstract}
In this note we consider the classical variational tools like: Ekelenad's Variational Principle, Mountain Pass Lemma and some of their corollaries subject to a parameter. Next, we investigate the behaviour of critical points obtained once a sequence of parameters is allowed to be convergent. Applications for the Dirichlet Boundary Value Problem on the Sierpi\'{n}ski Gasket are given in presence of assumptions which lead to fulfillment of the mountain geometry.
\end{abstract}

\vspace{0.3cm}
\noindent
\textbf{Keywords :} Ekeland's variational principle, Mountain Pass Theorem, problems on Sieprinski Gasket, dependence on parameters.\\
\vspace{0.3cm}\\ 
\noindent
\textbf{AMS Mathematics Subject Classification : } 49J40, 35J25

\section{Introduction}

In this paper we aim at providing the parametric versions of two well-known variational tools, namely the Ekeland's variational principle and the Mountain Pass Theorem. A multiplicity result based on these two tools will also be discussed. In our approach, we allow for the action functional to depend on some parameter and investigate what happens when this parameter approaches its limit. The main
question is whether corresponding to this limit we obtain a critical point (or critical points) provided that for any parameter from a seqeunce a critical point exists. Such abstract results have not been considered in the literature yet and we believe that these will also be applied to the so called variational stability or else to the continuous dependence on parameters for boundary value problems which are described for example in \cite{LedzewiczWalczak}. Moreover, we underline that for the limit problem we do not impose assumptions in the same manner as for any other problem from the sequence. This means that the variational tools mentioned above
cannot be directly applied.

Furthermore, we provide some applications for Laplacian problems considered on Sierpi\'{n}ski Gasket, namely we consider the following problem.: Let $V$ stand for the Sierpi\'{n}ski gasket, $V_{0}$ be its intrinsic boundary, let $\Delta $ denote the weak Laplacian on $V$ and let the measure $\mu $ denote the restriction to $V$ of normalized $\log N/\log 2-$dimensional Hausdorff measure, so that $\mu(V)=1$. Let $M>0$. In this paper we consider the existence of at least two nontrivial solutions to the following boundary value problem on $V$ for $m\in\naturals$:
\begin{equation}
\left\{ 
\begin{array}{l}
\Delta x(y)+a(y)x(y)+g_{m}\left( y\right) f(x(y))h\left( u_{m}\right)
=0\quad \text{for a.e. }y\in V\setminus V_{0}, \\ 
x|_{V_{0}}=0,%
\end{array}%
\right.   \label{rownanie}
\end{equation}

\noindent
where $(u_m)\subset L^{2}(V),\ (g_m) \subset C(V),\ |u_{m}(y)| \leq M$ for a.e. $y\in V$ and $f:\reals \rightarrow \reals$, $h:[-M,M] \rightarrow \reals_{+}$ are continuous functions. Solutions to (\ref{rownanie}) are understood in the weak sense which we will describe in more detail later. Using our abstract results we will examine what happens as $u_{m}\rightarrow u_{0}$ and $g_{m}\rightarrow
g_{0}$, where "$\rightarrow $" denotes norm convergence in the respective spaces.

We define $F:\reals\rightarrow\reals$ by $F(\xi)= \int_0^{\xi}\ f(x)\ dx$. Concerning the nonlinear term, we will employ the following conditions for every $m\in\naturals$:

\begin{description}
\item[A1] $a\in L^1(V,\mu)$ and $a\leq 0$ almost everywhere in 
$V$;

\item[A2] there exists positive constants $g^1,\ g^2$ such that $g^1\leq g_m(y) \leq g^2$ for all $y\in V$; there exists positive constants $h^1,\ h^2$ such that $h^1\leq h(v) \leq h^2$\ for all $v\in[-M,M]$;

\item[A3] there exist constants $\theta_m>2+\eps$, where $\eps>0$, such that for all $v\in\reals$ 
\begin{equation*}
0<\theta_mF(v) \leq vf(v).
\end{equation*}

\noindent
Moreover, there is a constant $c>0$\ such that $\frac{1}{2}-\frac{1}{\theta
_{m}}\geq c;$

\item[A4] there are constants $M_1>0$ and $\beta_m >0$ such that
\begin{equation*}
\max_{\substack{y\in V\\ |v| \leq M_1,|u|\leq M}}\ |g_{m}(y)f(v)h(u)| \leq 
\frac{M_{1}}{2 (\beta_{m}+1)(2N+3)^{2}};
\end{equation*}

\item[\textbf{A5}] for each $m$, there exists $\eta_{m}>\eta >0$\ such that for every $v\in[-1,1]$ we have
\begin{equation*}
F(v) \geq \eta_{m} |v|.
\end{equation*}
\end{description}

The Sierpi\'{n}ski gasket has its origin in a paper by Sierpi\'{n}ski \cite{Sie}. Our choice of this setting for the applications lies
mainly in Proposition 2.24 from \cite{FaHu}, which concerns checking the Palais-Smale condition. The Sierpi\'{n}ski Gasket is can be described as a subset of the plane obtained from an equilateral triangle by removing the open middle inscribed equilateral triangle of $1/4$ of the area, removing the corresponding open triangle from each of the three constituent triangles and continuing in this way.

The study of the Laplacian on fractals started in physical sciences in \cite{A} and \cite{Ra,Ra2}. The Laplacian on the Sierpi\'{n}ski gasket was first constructed in \cite{Ku} and \cite{gold}. Among the contributions to the theory of nonlinear elliptic equations on fractals we mention \cite{BRV,Fa99,FaHu} and \cite{K}, \cite{stribook}. Concerning some recent results by variational methods and critical point theory pertaining to the existence and the multiplicity of solutions by the recently developed variational tools we must mention the following sources \cite{bonanno1}, \cite{bonanno2}, \cite{BRaduV}, \cite{molica1}.

\section{Parametric theorems}

In what follows, $X$ stands for an arbitrary Banach space. With the advent of the Ekeland's Variational Principle (comp. \cite{Ekeland}), mathematicians gained an invaluable tool for the investigation of critical points. One such result can be found in \cite{jabri}, p. 27 (Corollary 3.4). 

\begin{theorem}
Let $\Phi : X \rightarrow \reals$  be a $C^1$-functional that is bounded from below and satisfies the Palais-Smale condition. Then there exists $\bar{x} \in X$ such that 
$$\Phi(\bar{x}) = \inf_{x\in X} \ \Phi(x) \hspace{0.4cm} \text{and} \hspace{0.4cm} \partial \Phi(\bar{x}) = 0.$$
\label{JabriCPT}
\end{theorem}

A natural question is the following: 
\begin{center}
\textit{What can we say about the limit of a sequence of functionals in Theorem \ref{JabriCPT}?}
\end{center}
 
\noindent
In a sense, we expect that if the sequence of functionals $(\Phi_n)$ behaves 'well-enough', then the \textit{limit functional} $\Phi$ will still have a critical point. Throughout the whole chapter, we suppose that a functional sequence $(\Phi_n) : X \rightarrow \reals$ is given. Moreover, we consider a functional $\Phi : X \rightarrow \reals$ which is in a sense (we will make it precise later) a limit of the aforementioned sequence. 

At first, we discuss the boundedness from below. The following lemma provides the condition for the limit functional to be bounded from below.

\begin{lemma}
Suppose that $(\Phi_n) : X \rightarrow \reals$ and $\Phi : X \rightarrow \reals$ satisfy the following:
\begin{description}
	\item[\hspace{0.4cm} (BB)] we have $-\infty < \inf_{\substack{x\in X \\ n\in\naturals}} \ \Phi_n(x)$, and
	\item[\hspace{0.4cm} (C)] the sequence $(\Phi_n)$ converges uniformly to $\Phi$, i.e.
	$$\forall_{\eps > 0} \ \exists_{N \in \naturals} \ \forall_{\substack{x \in X \\ n \geq N}} \ |\Phi(x) - \Phi_n(x)| < \eps.$$
\end{description}

\noindent
Then the functional $\Phi$ is bounded from below.  
\label{BB}
\end{lemma}
\begin{proof}
Fix $\eps > 0$. By \textbf{(C)}, there exists $N \in \naturals$ such that 
\begin{gather*}
\forall_{\substack{x \in X \\ n\geq N}} \ |\Phi(x) - \Phi_n(x)| < \eps,
\end{gather*}

\noindent
which implies that 
\begin{gather*}
\forall_{\substack{x \in X \\ n\geq N}} \ |\Phi_n(x)| - \eps < |\Phi(x)|.
\end{gather*}

\noindent
Taking the infimum on both sides, we obtain
$$- \infty \stackrel{\textbf{(BB)}}{<} \inf_{\substack{x \in X\\ n\geq N}} \ |\Phi_n(x)| - \eps \leq \inf_{x \in X}\ |\Phi(x)|,$$

\noindent
which ends the proof. 
\end{proof}

In order to apply Theorem \ref{JabriCPT} we need, apart from boundedness from below, some sort of Palais-Smale condition. The next result comes to our rescue. 

\begin{lemma}
Suppose that: 
\begin{description}
		\item[\hspace{0.4cm} (Diff)] $\Phi$ and $\Phi_n$ are $C^1$-functionals for every $n\in\naturals$. 
		\item[\hspace{0.4cm} (diagPS)] If $(x_n) \subset X$ is such that $\sup_{n \in\naturals} \ |\Phi_n(x_n)| < \infty$ and 
		$$\lim_{n\rightarrow\infty} \ \|\partial_1 \Phi_n(x_n)\| = 0,$$
		
		\noindent
		then there exists a convergent subsequence of $(x_n)$. 
		\item[\hspace{0.4cm} ($\partial$C)] $\Phi$ and $\partial_1 \Phi$ satisfy \emph{\textbf{(C)}}.
\end{description}

\noindent
Then the functional $\Phi$ satisfies Palais-Smale condition.
\label{diagPS}
\end{lemma}
\begin{proof}
Fix $\eps>0$ and suppose that $(x_n) \subset X$ is a sequence such that $\Phi(x_n)$ is bounded and \mbox{$\partial_1 \Phi(x_n) \rightarrow 0$}. Observe that by \textbf{($\partial$C)}, there exists $N \in \naturals$ such that 
\begin{gather*}
\forall_{n \geq N} \ |\Phi(x_n) - \Phi_n(x_n)| < \eps
\end{gather*} 

\noindent
and consequently
\begin{gather*}
\forall_{n \geq N} \ |\Phi_n(x_n)| < |\Phi(x_n)| + \eps.
\end{gather*}

\noindent
Taking the supremum on both sides, we obtain 
\begin{gather}
\sup_{n\geq N} \ |\Phi_n(x_n)| \leq \sup_{n\in\naturals} \ |\Phi(x_n)| + \eps < \infty.
\label{whataboutfirstelements}
\end{gather}

By \textbf{($\partial$C)} we know that there exists $N\in\naturals$ (possibly different than before) such that
\begin{gather*}
\forall_{n\geq N} \ \|\partial_1 \Phi_n(x_n) - \partial_1 \Phi(x_n)\| < \eps.
\end{gather*}

\noindent
Consequently, we have 
\begin{gather*}
\forall_{n \geq N} \ \|\partial_1 \Phi_n(x_n)\| \leq \eps,
\end{gather*}

\noindent
and at last:
\begin{gather}
\lim_{n\rightarrow\infty} \ \|\partial_1 \Phi_n(x_n)\| = 0.
\label{almostover}
\end{gather}

\noindent
By \textbf{(diagPS)}, we extract a convergent subsequence from $(x_n)$, which ends the proof. 
\end{proof}

At this point we are able to establish a critical point for the limit functional.

\begin{theorem}
Suppose that \emph{\textbf{(BB)}, \textbf{(Diff)}, \textbf{(diagPS)}} and \emph{\textbf{($\partial$C)}} are satisfied. Then there exists a point $\bar{x} \in X$ such that 
$$\Phi(\bar{x}) = \inf_{x\in X} \ \Phi(x) \hspace{0.4cm} \text{and} \hspace{0.4cm} \partial \Phi(\bar{x}) = 0.$$

\noindent
Moreover, if $(\bar{x}_n)$ is a sequence of critical points for $(\Phi_n)$ such that 
$$\sup_{n\in\naturals} \ \Phi_n(\bar{x}_n) < \infty \hspace{0.4cm} \text{and} \hspace{0.4cm} \inf_{x \in X}\ \Phi_n(x) = \Phi_n(\bar{x}_n),$$ 

\noindent
then a limit of a subsequence of $(\bar{x}_n)$ is a critical point of $\Phi$. 
\label{cptlambda}
\end{theorem}
\begin{proof}
Assumption \textbf{(Diff)} says that $\Phi$ is a $C^1$-functional. By \mbox{Lemma \ref{BB}} and \mbox{Lemma \ref{diagPS}} we know that $\Phi(\cdot,\bar{\lambda})$ is bounded from below and satisfies the Palais-Smale condition. Finally, using \mbox{Theorem \ref{JabriCPT}}, we obtain the existence of a critical point.

For the second part of theorem, by \textbf{(diagPS)} we know that some subsequence (we do not change the notation) of $(\bar{x}_n)$ converges. Fix $\eps >0$. By \textbf{(C)} we know that there exists $N \in \naturals$ such that 
$$\forall_{\substack{x \in X \\ n\geq N}} \ -\eps + \Phi_n(x) < \Phi(x) < \Phi_n(x) + \eps.$$

\noindent
Taking the infimum, for every $n\geq N$ we have 
\begin{gather*}
-\eps + \inf_{x \in X}\ \Phi_n(x) < \inf_{x \in X}\ \Phi(x) < \inf_{x \in X}\ \Phi_n(x) + \eps \\ 
\Longleftrightarrow \ -\eps + \Phi_n(\bar{x}_n) < \inf_{x \in X}\ \Phi(x) <  \Phi_n(\bar{x}_n) + \eps \\
\Longleftrightarrow \ |\inf_{x \in X}\ \Phi(x) -  \Phi_n(\bar{x}_n)| < \eps \\ 
\Longrightarrow \ |\inf_{x \in X}\ \Phi(x) -  \Phi(\bar{x}_n)| < \eps + |\Phi(\bar{x}_n) - \Phi_n(\bar{x}_n)| < 2\eps. 
\end{gather*}

\noindent
This means that 
$$\lim_{n\rightarrow\infty} \ \Phi(\bar{x}_n) = \Phi(\bar{x})$$

\noindent
which ends the proof. 
\end{proof}

In \cite{BerJebMawhin}, we can find (Proposition 1) the following variation on Theorem \ref{JabriCPT}:

\begin{theorem}
Let $\Phi : X \rightarrow \reals$ be a $C^1$-functional satisfying Palais-Smale condition. If there exists an open set $U \subset X$ such that 
$$-\infty < \inf_{x \in \overline{U}} \ \Phi(x) < \inf_{x \in \partial U} \ \Phi(x),$$

\noindent
then there exists $\bar{x} \in U$ such that 
$$\Phi(\bar{x}) = \inf_{x\in U} \ \Phi(x) \hspace{0.4cm} \text{and} \hspace{0.4cm} \partial \Phi(\bar{x}) = 0.$$
\label{CPTonU}
\end{theorem}

In order to rewrite Theorem \ref{CPTonU} in the parametric version, we prove the following result:

\begin{lemma}
Suppose that \emph{\textbf{(C)}} is satisfied and 
\begin{description}
	\item[\hspace{0.4cm} (InfU)] there exists an open set $U \subset X$ and $\eps, R > 0$ such that 
	$$ \forall_{n\in\naturals} \ -R < \inf_{x \in \overline{U}} \ \Phi_n(x) < \inf_{x \in \partial U} \ \Phi_n(x) - \eps.$$
\end{description}

\noindent
Then 
$$-\infty < \inf_{x \in \overline{U}} \ \Phi(x) < \inf_{x \in \partial U} \ \Phi(x).$$
\label{infulemma}
\end{lemma}
\begin{proof}
By \textbf{(C)} there exists $N>0$ such that 
\begin{gather}
\forall_{\substack{x\in X \\ n>N}} \ \Phi_n(x) - \frac{\eps}{2} < \Phi(x) < \Phi_n(x) + \frac{\eps}{2}.
\label{estimatefrombelowandabove}
\end{gather}

\noindent
Hence, we obtain 
\begin{gather*}
\forall_{n > N} \ -\infty < \inf_{x \in \overline{U}} \ \Phi_n(x) - \frac{\eps}{2} \stackrel{(\ref{estimatefrombelowandabove})}{\leq} \inf_{x \in \overline{U}} \ \Phi(x) 
\stackrel{(\ref{estimatefrombelowandabove})}{\leq} \inf_{x \in \overline{U}} \ \Phi_n(x) + \frac{\eps}{2} \\
\stackrel{\textbf{(InfU)}}{<} \inf_{x \in \partial U} \ \Phi_n(x) - \frac{\eps}{2} \stackrel{(\ref{estimatefrombelowandabove})}{\leq} \inf_{x\in\partial U} \ \Phi(x),
\end{gather*}

\noindent
which ends the proof. 
\end{proof}

\begin{theorem}(parametric critical point theorem on open set)\\
Let \emph{\textbf{(Diff)}, \textbf{(diagPS)}, \textbf{($\partial$C)}} and \emph{\textbf{(InfU)}} be satisfied. Then there exists a point $\bar{x} \in U$ such that 
$$\Phi(\bar{x}) = \inf_{x\in \overline{U}}\ \Phi(x) \hspace{0.4cm} \text{and} \hspace{0.4cm} \partial \Phi(\bar{x}) = 0.$$
\label{cptonU}
\end{theorem}
\begin{proof}
Assumption \textbf{(Diff)} says that $\Phi$ is a $C^1$-functional. By \mbox{Lemma \ref{diagPS}} and \mbox{Lemma \ref{infulemma}} we know that $\Phi$ satisfies the Palais-Smale condition and 
$$-\infty < \inf_{x \in \overline{U}} \ \Phi(x) < \inf_{x \in \partial U} \ \Phi(x).$$

\noindent
Finally, by Theorem \ref{CPTonU}, we obtain the desired critical point.  
\end{proof}

Another result, alongside the Ekeland's Variational Principle, which is extremely useful in critical point theory, is the celebrated Mountain Pass Theorem (comp. \cite{Ambrosetti}). We recall the version which comes from \cite{jabri}, p. 66:

\begin{theorem}
Let $\Phi : X \rightarrow \reals$ be a $C^1$-functional satisfying Palais-Smale condition. Suppose that there is a nonzero element $x_{\ast} \in X$ and $r > 0$ with $r < \|x_{\ast}\|$ such that 
\begin{gather*}
\max(\Phi(0),\Phi(x_{\ast})) < \inf_{x \in S(0,r)} \ \Phi(x).
\end{gather*}

\noindent
Then the functional $\Phi$ has a critical value $\bar{y} \geq \inf_{x \in S(0,r)} \ \Phi(x)$ characterized by
$$\bar{y} = \inf_{\gamma \in \Gamma} \ \sup_{x \in \gamma([0,1])} \ \Phi(x),$$

\noindent
where 
$$\Gamma = \bigg\{\gamma \in C([0,1],X) \ : \ \gamma(0)=0, \ \gamma(1) = x_{\ast}\bigg\}.$$ 
\label{MPT}
\end{theorem}

Here is our version of the above classical result:

\begin{theorem}(Parametric Mountain Pass Theorem)\\
Suppose that \emph{\textbf{(Diff)}, \textbf{(diagPS)}, \textbf{(C)}} are satisfied. Moreover, assume that:
\begin{description}
	\item[\hspace{0.4cm} (PMPT1)] $\Phi_n(0) = 0$ for every $n\in\naturals$,
	\item[\hspace{0.4cm} (PMPT2)] there exists $r > 0$ such that 
	$$\inf_{\substack{x \in S(0,r) \\ n \in \naturals}} \ \Phi_n(x) > 0,$$ 
	
	\item[\hspace{0.4cm} (PMPT3)] there exists $x_*$ such that $\|x_*\| > r$ and $\sup_{n\in\naturals} \ \Phi_n(x_*) < 0$.
\end{description}

\noindent
Then there exists a point $\bar{x} \in X$ such that 
$$\Phi(\bar{x}) \geq \inf_{x \in S(0,r)}\ \Phi(x) > 0 \hspace{0.4cm} \text{and} \hspace{0.4cm} \partial_1 \Phi(\bar{x}) = 0.$$
\label{ParametricMountainPassTheorem}
\end{theorem}
\begin{proof}
By \textbf{(PMPT1)} and \textbf{(C)}, it is easy to see that that $\Phi(0) = 0$. If we put 
$$\eps := \frac{1}{2}\inf_{\substack{x \in S(0,r) \\ n\in\naturals}} \ \Phi_n(x),$$

\noindent
then by \textbf{(C)} there exists $N>0$ such that 
$$\forall_{\substack{x\in X \\ n > N}} \ \Phi_n(x) - \eps < \Phi(x).$$

\noindent
As a result, we have $\eps < \inf_{x\in S(0,r)} \ \Phi(x)$.

We now put 
$$\eps := -\sup_{n\in\naturals} \ \Phi_n(x_*),$$

\noindent
which is positive by \textbf{(PMPT3)}. By \textbf{(C)} there exists $N>0$ (possibly different than before) such that 
$$\forall_{n > N} \ \Phi(x_*) < \Phi_n(x_*) + \eps.$$

\noindent
Taking the supremum, we conclude that $\Phi(x_*) \leq 0$.

Finally, we observe that $\Phi$ satisfies the assumptions of the classical Mountain Pass Theorem. As a consequence, we obtain the desired critical point. 
\end{proof}

Finally, we are able to rewrite Proposition 2 from \cite{BerJebMawhin}, which reads as follows:

\begin{theorem}
Let $\Phi : X \rightarrow \reals$ be a $C^1$-functional satisfying Palais-Smale condition. Suppose that there is nonzero element $x_* \in X$ and $r > 0$ such that
\begin{itemize}
	\item $\Phi(0) = 0$ and 
	$$-\infty < \inf_{x \in \overline{B}(0,r)} \ \Phi(x) < 0 < \inf_{x \in S(0,r)} \ \Phi(x),$$
	
	\item  $\Phi(x_*) \leq 0$ and $\|x_*\| > r$.
\end{itemize}

\noindent
Then $\Phi$ has at least two nontrivial critical points. 
\label{BJM}
\end{theorem}

\begin{theorem}(Double Critical Point Theorem)\\
Suppose that \emph{\textbf{(Diff)}, \textbf{(diagPS)}, \textbf{(C)}} are satisfied. Moreover, assume that:
\begin{description}
	\item[\hspace{0.4cm} (DCPT1)] $\Phi_n(0) = 0$ for every $n\in\naturals$,
	\item[\hspace{0.4cm} (DCPT2)] there exists $r, R > 0$ such that 
	$$\forall_{n\in\naturals} \ - R < \inf_{x \in \overline{B}(0,r)} \ \Phi_n(x) < 0 < \inf_{x \in S(0,r)} \ \Phi_n(x),$$ 
	
	\item[\hspace{0.4cm} (DCPT3)] there exists $x_*$ such that $\|x_*\| > r$ and $\sup_{n\in\naturals} \ \Phi_n(x_*) < 0$.
\end{description}

\noindent
Then there exist two nontrivial critical points.
\label{doubleCPt}
\end{theorem}
\begin{proof}
Looking at Lemma \ref{infulemma} and the Parametric Mountain Pass Theorem, \textbf{(DCPT2)} implies that
$$-\infty < \inf_{x \in \overline{B}(0,r)} \ \Phi(x) \hspace{0.4cm} \text{and} \hspace{0.4cm} 0 < \inf_{x \in S(0,r)} \ \Phi(x).$$

\noindent
It remains to prove that $\inf_{x \in \overline{B}(0,r)} \ \Phi(x) < 0$. Then, applying Theorem \ref{BJM} for $\Phi$, we obtain two nontrivial critical points. 

We put 
$$\eps := -\frac{1}{2}\inf_{\substack{x \in \overline{B}(0,r) \\ n\in\naturals}} \ \Phi_n(x),$$

\noindent
which is positive by \textbf{(DCPT2)}. By \textbf{(C)} we know that there exists $N \in\naturals$ such that 
$$\forall_{\substack{x \in X \\ n \geq N}} \ \Phi(x) < \Phi_n(x) + \eps.$$

\noindent
This implies that
$$\inf_{x \in \overline{B}(0,r)} \ \Phi(x) < - \eps < 0,$$
  
\noindent
which ends the proof. 
\end{proof}

\section{Remarks on the abstract fractal setting}

Concerning the Sierpi\'{n}ski gasket we follow remarks collected in \cite{bonanno1}. Let $N\geq 2$ be a natural number and let $p_{1},\dots ,p_{N}\in \reals^{N-1}$ be so that $|p_{i}-p_{j}|=1$ for $i\neq j$. Define, for every $i\in\{1,\dots,N\}$, the map $S_{i}:\reals^{N-1}\rightarrow \reals^{N-1}$ by 
\begin{equation*}
S_{i}(x)=\frac{1}{2}x+\frac{1}{2} p_{i}.
\end{equation*}%
Let $S:=\{S_{1},\dots,S_{N}\}$ and denote by $G :P(\reals^{N-1})\rightarrow P(\reals^{N-1})$ the map assigning to a subset $A$ of $\reals^{N-1}$ the set 
\begin{equation*}
G(A)=\bigcup_{i=1}^{N}S_{i}(A).
\end{equation*}

\noindent
It is known that there is a unique non-empty compact subset $V$ of $\reals^{N-1}$, called the attractor of the family $S$, such that $G(V)=V$; see, Theorem 9.1 in \cite{Fa}.

The set $V$ is called the Sierpi\'{n}ski gasket in $\reals^{N-1}$. It can be constructed inductively as follows: Put $V_0:=\{p_{1},\dots ,p_{N}\}$ which is called the intrinsic boundary of $V$ and define $V_{m}:=G(V_{m-1})$ for $m\geq 1$. Put $V_*:=\bigcup_{m\geq 0}\ V_{m}$. Since $p_{i}=S_{i}(p_{i})$ for $i\in \{1,...,N\}$, we have $V_{0}\subseteq V_{1}$, hence $G(V_*)=V_*$. Taking into account that the maps $S_{i},\ i\in \{1,...,N\}$, are homeomorphisms, we conclude that $\overline{V_*}$ is a fixed point of $G$. On the other hand, denoting by $C$ the convex hull of the set $\{p_{1},\ldots,p_{N}\}$, we observe that $S_{i}(C)\subseteq C$ for $i=1,...,N$. Thus $V_m \subseteq C$ for every $m\in N$, so $\overline{V_*}\subseteq C$. It follows that $\overline{V_*}$ is non-empty and compact, hence $V=\overline{V_*}$.

We endow $V$ with the relative topology induced from the Euclidean topology on $\reals^{N-1}$. By $C(V)$, we denote the space of real-valued continuous functions on $V$ and by 
\begin{equation*}
C_{0}(V):=\bigg\{u\in C(V)\ :\ u|_{V_{0}}=0\bigg\}.
\end{equation*}

\noindent
The spaces $L^2(V,\mu),\ C(V)$ and $C_{0}(V)$ are endowed with the usual norms, i.e. the norm induced by the inner product
\begin{equation*}
\langle v|h\rangle = \int_{V}\ v(y)h(y)\ d\mu 
\end{equation*}

\noindent
and supremum norm $\|\cdot\|_{\infty }$, respectively.

For a function $u:V\rightarrow \reals$ and for $m\in N$ let 
\begin{equation}
W_{m}(u) := \left(\frac{N+2}{N}\right)^{m}\ \sum_{\underset{|x-y|=2^{-m}}{%
x,y\in V_{m}}}\ (u(x)-u(y))^{2}.  
\label{defWm}
\end{equation}

\noindent
Since $W_{m}(u)\leq W_{m+1}(u)$ for every natural $m$, we can put 
\begin{equation*}
W(u)=\lim_{m\rightarrow \infty}\ W_{m}(u).
\end{equation*}

\noindent
Define now 
\begin{equation*}
H_{0}^{1}(V):=\bigg\{u\in C_{0}(V)\ :\ W(u)<\infty \bigg\}.
\end{equation*}

\noindent
This space is a dense linear subset of $L^{2}(V,\mu)$ equipped with the $\|\cdot\|_{2}-$norm. We endow $H_{0}^{1}(V)$ with the norm 
\begin{equation*}
\|u\|=\sqrt{W(u)}.
\end{equation*}

\noindent
There is an inner product defining this norm: for $u,v\in H_{0}^{1}(V)$ and $m\in N$ let 
\begin{equation*}
{\mathcal{W}}_{m}(u,v)=\left(\frac{N+2}{N}\right)^{m}\ \sum_{\underset{|x-y|=2^{-m}}{x,y\in V_{m}}}(u(x)-u(y))(v(x)-v(y)).
\end{equation*}

Put 
\begin{equation*}
\W(u,v)=\lim_{m\rightarrow \infty }\ \W_{m}(u,v).
\end{equation*}

\noindent
The space $H_{0}^{1}(V)$, equipped with the inner product $W$ inducing the norm $\|\cdot\|$, becomes a real Hilbert spaces. Moreover, 
\begin{equation}
\|u\|_{\infty} \leq (2N+3)\|u\|,\hspace{0.4cm} \text{for every} \hspace{0.4cm} u\in
H_{0}^{1}(V),  
\label{embeddingconstant}
\end{equation}

\noindent
and the embedding 
\begin{equation*}
(H_{0}^{1}(V),\|\cdot\|) \hookrightarrow (C_{0}(V),\|\cdot\|_{\infty})
\end{equation*}

\noindent
is compact, see also, \cite{FuSc} for further details.

Note that $(H_{0}^{1}(V),\|\cdot\|)$ is a Hilbert space which is dense in $L^2(V,\mu)$. Furthermore, $W$ is a Dirichlet form on $L^{2}(V,\mu)$. Let $Z$ be a linear subset of $H_{0}^{1}(V)$, which is dense in $L^{2}(V,\mu)$. Then, in \cite{FaHu}, we find a linear self-adjoint operator $\Delta:Z\rightarrow L^{2}(V,\mu)$, the (weak) Laplacian on $V$, given by
\begin{equation*}
-\W(u,v)=\int_{V}\ \Delta u\cdot v\ d\mu, \hspace{0.4cm}\text{for every}\hspace{0.4cm} (u,v)\in Z\times H_{0}^{1}(V).
\end{equation*}

Let $H^{-1}(V)$ be the closure of $L^{2}(V,\mu )$ with respect to the
pre-norm 
\begin{equation*}
\|u\|_{-1}=\sup_{\substack{\|h\|=1\\ h\in H_{0}^{1}(V)}}\ |\langle u|h\rangle|,
\end{equation*}

\noindent
where $v\in L^{2}(V,\mu)$ and $h\in H_{0}^{1}(V)$. Then $H^{-1}(V)$ is a Hilbert
space and the relation 
\begin{equation*}
\forall_{v\in H_{0}^{1}(V)}\ -\W(u,v)=\langle \Delta u|v\rangle
\end{equation*}

\noindent
uniquely defines a function $\Delta u\in H^{-1}(V)$ for every $u\in H_{0}^{1}(V)$.

\section{Applications}

We observe that by (\ref{embeddingconstant}), for every $y\in V$
\begin{equation}
|x(y)|\leq \|x\|_{\infty }\leq (2N+3)\|x\|_{H_{0}^{1}(V)}.
\label{nier_C_H01}
\end{equation}

\noindent
Using the first inequality in (\ref{nier_C_H01}) and the fact that $\mu(V) =1$ we get
\begin{equation*}
\|x\|_{L^{2}(V,\mu)} \leq \|x\|_{\infty} \leq (2N+3) \|x\|_{H_{0}^{1}(V)}
\end{equation*}

\noindent
for any $x\in H_{0}^{1}(V)$.

In what follows, we fix $m$ and $u_m$. We say that a function $x\in
H_{0}^{1}(V)$ is called a weak solution of (\ref{rownanie}), if
 
\begin{equation*}
\W(x,v)-\int_{V}\ a(y)x(y)v(y)\ d\mu - \int_{V}\ g_{m}(y) f(x(y)) h(u_{m}(y)) v(x)\ d\mu =0,
\end{equation*}

\noindent
for every $v\in H_{0}^{1}(V)$. Consequently, whenever we write that we obtain a
solution to (\ref{rownanie}), we mean the weak one. The functional $J_{m} : H_{0}^{1}(V)\rightarrow \reals$, given by 
\begin{equation}
J_{m}(x) = \frac{1}{2}\|x\|^{2} - \frac{1}{2}\int_{V}\ a(y)x^{2}(y)\ d\mu
-\int_{V}\ g_{m}(y) F(x(y))h(u_{m}(y))\ d\mu 
\label{energyfunctional}
\end{equation}

\noindent
for every $x\in H_{0}^{1}(V)$, is the Euler action functional attached to the problem (\ref{rownanie}).

\begin{lemma}
Assume that \textbf{A1}, \textbf{A2} holds. Then, the functional $J_{m}:H_{0}^{1}(V)\rightarrow \reals$ defined by the relation (\ref{energyfunctional}) is a $C^{1}(H_{0}^{1}(V),\reals)-$functional. Moreover, 

\begin{equation*}
\forall_{w\in H_{0}^{1}(V)}\ J_{m}'(x)(w)=\W(u,w)-\int_{V}\ a(y)x(y)w(x)\ d\mu -\int_{V}\ g_{m}(y)f(x(y))h(u_{m}(y))\ d\mu,
\end{equation*}

\noindent
for each point $x\in H_{0}^{1}(V)$. In particular, $x\in H_{0}^{1}(V)$ is a
weak solution of problem (\ref{rownanie}) if and only if $x$ is a critical
point of $J_{m}$; $J_{m}$ is also weakly l.s.c.
\label{criticalpoint}
\end{lemma}

We note that assumptions \textbf{A1}-\textbf{A4} lead to the existence of a solution to (\ref{rownanie}) for any value of the parameter $m$ by the Mountain Pass Theorem as suggested in \cite{FaHu}.

\subsection{Application of the Parametric Mountain Pass Theorem}

Now, we apply our parametric results to problem (\ref{rownanie}). We start with the applications of the Parametric Mountain Pass Theorem .

\begin{theorem}
Assume that \textbf{A1}-\textbf{A4} are satisfied. Let $g_{m}\rightarrow g_{0}$ in $C(V)$ and let $u_{m}\rightarrow u_{0}$ in $L^{2}(V)$. Then the problem 
\begin{equation}
\left\{ 
\begin{array}{l}
\Delta x(y)+a(y)x(y)+g_{0}\left( y\right) f(x(y))h\left( u_{0}\left(
y\right) \right) =0\quad \text{for a.e. }y\in V\setminus V_{0}, \\ 
x|_{V_{0}}=0,%
\end{array}%
\right.   
\label{limit_MPT_problem}
\end{equation}

\noindent
has a nontrivial solution.
\end{theorem}
\begin{proof}
We need to demonstrate that the assumptions of Theorem \ref{ParametricMountainPassTheorem} are satisfied. Some calculations are taken from \cite{FaHu} and repeated here for the Reader's convenience. We believe, it makes our proof more clear.

From \cite{FaHu} p. 563, we see that \textbf{(PS$\lambda$)} holds. Indeed, suppose
that $|J_{m}(x_k)| \leq b$ for all $m,k$ and that 
\begin{equation*}
\lim_{m\rightarrow \infty}\ \lim_{k\rightarrow \infty}\ \|J_{m}'(x_k)\|_{H_{0}^{-1}(V)}=0.
\end{equation*}

\noindent
We see that 
\begin{equation*}
b+\frac{1}{2+\eps} \|x_k\| \geq b+\frac{1}{\theta_{m}} \|x_k\| \geq \left( \frac{1}{2}-\frac{1}{\theta _{m}}\right) \|x_k\|^{2} \geq c\|x_{k}\|^2.
\end{equation*}

\noindent
By Proposition 2.24 from \cite{FaHu} which says
that if a Palais Smale sequence is bounded then it is convergent, possibly up to a
subsequence, we have our assertion. 

Conditions \textbf{(Diff)} and \textbf{($\partial C$)} are obviously satisfied by continuity. Condition \textbf{(PMPT1)} is satisfied by definition of $J_m$. We fix a ball $B$ such that $|v(y)| \leq M_1$ for any $v\in B$. Concerning \textbf{(PMPT2)}, we see that it follows from formula (3.8) \cite{FaHu}, p. 562. Indeed, a function $x \mapsto \frac{x}{1+x}$ is bounded from below by some $\frac{\beta}{1+\beta}$ on $[\beta,+\infty)$. From the formula we mentioned earlier, we have
\begin{equation*}
J_m(u) \geq \frac{\beta}{1+\beta} \frac{M_0^2}{2(2N+3)}.
\end{equation*}

As for \textbf{(PMPT3)} note that there exists positive constants $b_1,\ b_2$ such that for all $v$:
\begin{equation*}
F(v) \geq b_1 |v|^{\theta_{m}}-b_2.
\end{equation*}

\noindent
Let us fix $x\in H_{0}^{1}(V)$ such that $|x(y)| \geq 1.$ Then, for $s>0$ we have
\begin{equation*}
\int_V\ g_{m}(y) F(sx(y))h(u_{m}(y))\ d\mu \geq b_1 g^1 h^1 s^{\theta_{m}}-b_2.
\end{equation*}%

\noindent
Therefore, for $s>1$ 
\begin{equation*}
J_m(sx) \leq \frac{1}{2}\left( \|x\| - \int_V\ a(y)x^{2}(y)\ d\mu \right) s^{2}-\alpha s^{2+\eps}+b_2,
\end{equation*}

\noindent
where $\alpha = b_1 g^1 h^1$. Since $\eps >0,\ a\leq 0$, then there exists $s^*$ such that 
\begin{equation*}
\frac{1}{2}\left( \|x\| - \int_V\ a(y)x^{2}(y)\ d\mu \right)(s^*)^{2}-\alpha(s^*)^{2+\eps}+b_2<0.
\end{equation*}

\noindent
Consequently, this implies the existence of $x^*$, independent of $m$ and outside of $B$, which
satisfies 
\begin{equation*}
J_{m}\left( x^{\ast }\right) <0\text{.}
\end{equation*}%
Now our assertion follows by Theorem 9.
\end{proof}

\subsection{Results for the double critical point theorem}

\begin{theorem}
Assume that A1-A5 are satisfied. Let $g_{m}\rightarrow g_{0}$ in $C\left(
V\right) $ and let $u_{m}\rightarrow u_{0}$ in $L^{2}\left( V\right) $. Then
problem (\ref{limit_MPT_problem}) has at least two nontrivial solutions.
\end{theorem}

We need only verify condition (DCPT2). Let us fix $x\in H_{0}^{1}\left(
V\right) $ such that $\left\Vert x\right\Vert \leq (2N+3)^{-1}$ and $\frac{1%
}{2}1\leq \left\vert x(y)\right\vert \leq 1$ for all $y\in V$. By A5 we see
that for $s>0$%
\begin{equation*}
\int_{V}g_{m}\left( y\right) F(sx(y))h\left( u_{m}\left( y\right) \right)
d\mu \geq \eta _{m}s\int_{V}\left\vert x(y)\right\vert g_{m}\left( y\right)
h\left( u_{m}\left( y\right) \right) d\mu \geq \alpha s\text{,}
\end{equation*}%
where $\alpha =\frac{1}{2}\eta g^{1}h^{1}$. Put 
\begin{equation*}
\tau =\left( \Vert x\Vert ^{2}-\int_{V}a(y)x^{2}(y)d\mu \right) >0.
\end{equation*}%
Thus 
\begin{equation*}
J_{m}\left( sx\right) =\frac{1}{2}s^{2}\tau -\alpha s.
\end{equation*}%
Therefore for $s$ small enough $J_{m}\left( sx\right) <\gamma <0$, where $%
\gamma $ does not depend on $m.$

\begin{acknowledgement}
This research of the first author was supported by grant no.
2014/15/B/ST8/02854 "Multiscale, fractal, chemo-hygro-thermo-mechanical
models for analysis and prediction the durability of cement based composites"
\end{acknowledgement}

\begin{tabular}{l}
Marek Galewski \\ 
Institute of Mathematics, \\ 
Technical University of Lodz, \\ 
Wolczanska 215, 90-924 Lodz, Poland, \\ 
marek.galewski@p.lodz.pl \\
\vspace{0.3cm}\\
Mateusz Krukowski \\
Institute of Mathematics, \\ 
Technical University of Lodz, \\ 
Wolczanska 215, 90-924 Lodz, Poland, \\
krukowski.mateusz13@gmail.com
\end{tabular}

\end{document}